\documentclass[12pt]{article}

\usepackage{amsthm}
\usepackage{amssymb}
\usepackage[pdftex]{graphicx}
\usepackage{epstopdf}

\newtheorem{theorem}{\large \it Theorem}
\newtheorem{lemma}[theorem]{\it Lemma}
\newtheorem{definition}{Definition}

\usepackage{geometry}
\begin{document}

\title{Characteristics of shape and knotting in ideal rings. }

\author{Laura Zirbel\\
  Department of Mathematics,\\
  University of California, Santa Barbara,\\
  Santa Barbara, CA 93016, USA\\
  Email: \texttt{lzirbel@math.ucsb.edu}\\
  Homepage: \texttt{http://www.math.ucsb.edu/~lzirbel}
\and Kenneth C. Millett\\
  Department of Mathematics,\\
  University of California, Santa Barbara,\\
  Santa Barbara, CA 93016, USA\\
  Email: \texttt{millett@math.ucsb.edu}\\
  Homepage: \texttt{http://www.math.ucsb.edu/~millett}}

\maketitle

\begin{abstract}
We present two descriptions of the the local scaling and shape of ideal rings, primarily featuring subsegments. Our focus will be the squared radius of gyration of subsegments and the squared internal end to end distance, defined to be the average squared distance between vertices $k$ edges apart. We calculate the exact averages of these values over the space of all such ideal rings, not just a calculation of the order of these averages, and compare these to the equivalent values in open chains. This comparison will show that the structure of ideal rings is similar to that of ideal chains for only exceedingly short lengths.

These results will be corroborated by numerical experiments. They will be used to analyze the convergence of our generation method and the effect of knotting on these characteristics of shape.     
\end{abstract}

\section{Introduction}

Long strings of connected molecules, called polymers, are central structures in life and physical sciences, as well as engineering. Prominent examples are DNA, proteins, polystyrene, and silicone. With regard to DNA, Fiers and Sinsheimer first showed that the DNA of a specific virus is a single-stranded ring \cite{Fiers62}. Because of this closure condition, DNA, like other closed polymers, can be knotted. In 1976, Liu et al. discovered examples of knotted DNA, which was followed by the discovery of topoisomerases, enzymes which knot and unknot DNA \cite{Cham01, Liu76}. These discoveries suggest that knotting plays an important role in the behavior and shape of DNA and other polymers.

We will examine ideal rings, embedded equilateral polygons in $\mathbb{R}^3$, which provide a model for polymers under the $\theta$-condition, where excluded volume can be ignored \cite{degennes, flory, suzuki2009}. We will compare ideal rings with ideal chains,  random walks in $\mathbb{R}^3$, to determine the effect of the closure constraint as well as the effect of knotting. 

These models will be used to analyze the relationship between shape, scale, knotting and two physical characteristics specific to  their local structure. We define the average squared radius of gyration of subsegments, calculated by averaging the standard squared radius of gyration of each subsegment of length $k$, and the squared internal end to end distance, the average distance between vertices $k$ edges apart. Rather than finding approximations of these averages, we will determine the exact theoretical averages for these descriptions of shape, as well as discuss numerical simulations and show how these characteristics are affected by knotting.

In numerical studies examining squared end to end distance, squared end to end distance is refered to internal end to end distance or two point correlation \cite{shimamura2005, valle2005, yao2004}. These characteristics are simplified in open chains because subsegments of length $k$ in chains of length $100$ are identical to subsegments of length $k$ in chains of length $100,000$: the ambient length has no effect on the behavior of subsegments. However, in the case of ideal rings, subsegments of length $k$ in a ring of length $n$ do not behave the same as a subsegment of length $k$ in a ring of length $m \neq n$. Thus, subsegment behavior relies on both $k$ and $n$. It will be shown that the structure of ideal rings is similar to that of ideal chains for only exceedingly short lengths, which may correspond with the difference in $\theta$-temperature between open chain polymers and ring polymers \cite{suzuki2009}. Understanding in the case of closed chains may also allow us to use these charactereistics to identify and describe the knotted portions of open chains \cite{kanaeda2008, shimamura2005}.

\section{An Introduction to Ideal Rings and Their Notation}
 
 \begin{definition} An {\em ideal ring}, $P$, is an $n$-edged equilateral polygon embedded in $\mathbb{R}^3$, with one vertex at the origin. Let $e_1,...,e_n$ be unit vectors such that the $k^{\textrm{th}}$ vertex of $P$ is $v_k = \displaystyle \sum_{i=1}^k e_i$. Let each $e_i$ be called an {\em edge vector} of $P$. The fact that $P$ is a polygon is equivalent to requiring closure, that is, that $v_n = \displaystyle \sum_{i=1}^n e_i = 0.$ Let $\mathcal{P}_n$ denote the space of all such polygons.
\end{definition}

The careful reader may note that these ideal rings are more specific than usual. We require these polygons to be based at the origin, and be oriented. This definition makes notation concise, and does not affect the averages we will calculate, as they are independent of the base point and orientation. 

For each $P \in \mathcal{P}_n$, the edge vectors $e_i$ are all identically distributed. Therefore, when $i \neq j$, $(e_i \cdot e_j)$ is a random variable that does not depend on $i$ and $j$. 

We will compare these with ideal chains, a similar population but without the closure constraint. 

 \begin{definition} An {\em open chain}, $W$, also sometimes called an {\em ideal chain} or {\em ideal open chain}, is an $n$-edged random walk in $\mathbb{R}^3$, where each edge has unit length. Let $e_1,...,e_n$ be unit vectors such that the $k^{\textrm{th}}$ vertex is given by $v_k = \displaystyle \sum_{i=1}^k e_i$. As with ideal rings, let $e_i$ be called an edge vector. Let $\mathcal{W}_n$ denote the space of all such open chains.
\end{definition}

Again, the careful reader will see that these are also based at the origin for ease of notation.

\section{Rigorous Calculation of Theoretical Averages for Ideal Rings and Ideal Chains}

\subsection{Squared End to End Distance in Ideal Chains}

For comparison, we will first calculate the squared end to end distance in ideal chains, as in \cite{DoiEdwards}. 

\begin{lemma}
\label{Cn}
In $\mathcal{W}_n$, the average squared end to end distance is $n$. 
\end{lemma}

\begin{proof}
Let $W \in \mathcal{W}_n$. Calculating the end to end distance of $W$: 

\begin{eqnarray*}
\left|\left| \displaystyle \sum_{i=1}^n e_i \right|\right|^2 & = & (\sum_{i=1}^n e_i \cdot \sum_{i=1}^n e_i) \\
& = & \sum_{i = 1}^n (e_i \cdot e_i) + 2 \left( \sum_{i = 1}^{n-1} \sum_{j = 1}^n (e_i \cdot e_j) \right) \\
\end{eqnarray*}

For random walks $W \in \mathcal{W}_n$, the direction of each edge vector is completely uncorrelated with the direction of the previous edge vector, and each edge vector $e_i$ is a uniformly distributed, unit length, random vector. Thus, for $i = j$, $(e_i \cdot e_j) = 1$, and for $i \neq j$ the average value of $(e_i \cdot e_j)$,  $\left< e_i \cdot e_j \right>$, is $0$. Thus we have that the average end to end distance, taken over all $W \in \mathcal{W}_n$, is $$\left<\left|\left| \displaystyle \sum_{i=1}^n e_i \right|\right|^2\right> = n.$$ \\

\end{proof}

\subsection{Average Edge Product in Rings}

Now we will consider ideal rings. As noted before, $(e_i \cdot e_j)$ is independent of $i$ and $j$. Here, we find an average value for $(e_i \cdot e_j)$, taken over all $i$, $j$ for all $P \in \mathcal{P}_n$. 

\begin{definition}
Consider the space of ideal rings, $\mathcal{P}_n$, for some $n$. Let $r_n$ denote the average of the set $\mathcal{R}_n$, $$\displaystyle \mathcal{R}_n = \{(e_i\cdot e_j) : e_i,e_j \textrm{ are edge vectors of some } P \in \mathcal{P}_n, i \neq j\}.$$ We will call {\em $r_n$ the average edge product.} 
\end{definition}

The following Lemma is generally known, and is foundational to the following proofs \cite{Grosberg08}. 

\begin{lemma}
For all $n\in \mathbb{N}$, the average edge product, over all ideal rings of length $n$, is given by $\displaystyle r_n = \frac{-1}{n-1}$. 
\end{lemma}

\begin{proof}
For all $P \in \mathcal{P}_n$, $\displaystyle \sum_{i=1}^n e_i = 0$. Then squaring both sides we have 
\begin{eqnarray*}
(0 \cdot 0) = 0& = & (\displaystyle \sum_{i=1}^n e_i) \cdot (\displaystyle \sum_{i=1}^n e_i ) \\
& = & \sum_{i=1}^n(e_i\cdot e_i)+2\sum_{1\leq i < j \leq n} (e_i \cdot e_j). 
\end{eqnarray*}

Taking the average over all $P \in \mathcal{P}$, we replace $(e_i \cdot e_j)$ with $\left<e_i \cdot e_j\right> = r_n$:

$$0 = n+ 2 \frac{n(n-1)}{2} r_n.$$ 

Solving for $r_n$ we have $$r_n = \frac{-n}{n(n-1)} = \frac{-1}{n-1}.$$
\end{proof}

\subsection{Squared End to End Distance}

In open chains, the average squared end to end distance of a subsegments of length $k$ in a chain of length $n$ is identical to the average squared end to end distance of chains of length $k$. That is, the ambient chain has no effect on the shape of the subsegment. The same is not true in ideal rings, where the closure constraint plays a pivotal role in local scale, as we will see.  

\begin{definition} For any $P \in \mathcal{P}_n$, define $$d_{P}(k,j) = \displaystyle \left|\left|\sum_{i=j}^{j+k} e_i \right|\right|^2 = ( \sum_{i=j}^{j+k} e_i)\cdot( \sum_{i=j}^{j+k} e_i),$$ the squared distance between the $j^{\textrm{th}}$ and $j+k^{\textrm{th}}$ vertex. We call $d_{P}(k,j)$ the {\em squared end-to-end distance of length $k$ at $j$} of $P$. 
Define $A(k,n)$ to be the average value of $$\{d_{P}(k,j): 1\leq j \leq n, P \in \mathcal{P}_n\},$$ that is, let $A(k,n)$ denote the average squared end to end distance of a subsegment of length $k$ in an ideal ring of length $n$. We will call $A(k,n)$ the {\em average end to end distance of length $k$}. 
\end{definition}

\begin{figure}[!htp]
  \begin{center}
    \includegraphics[scale=.36]{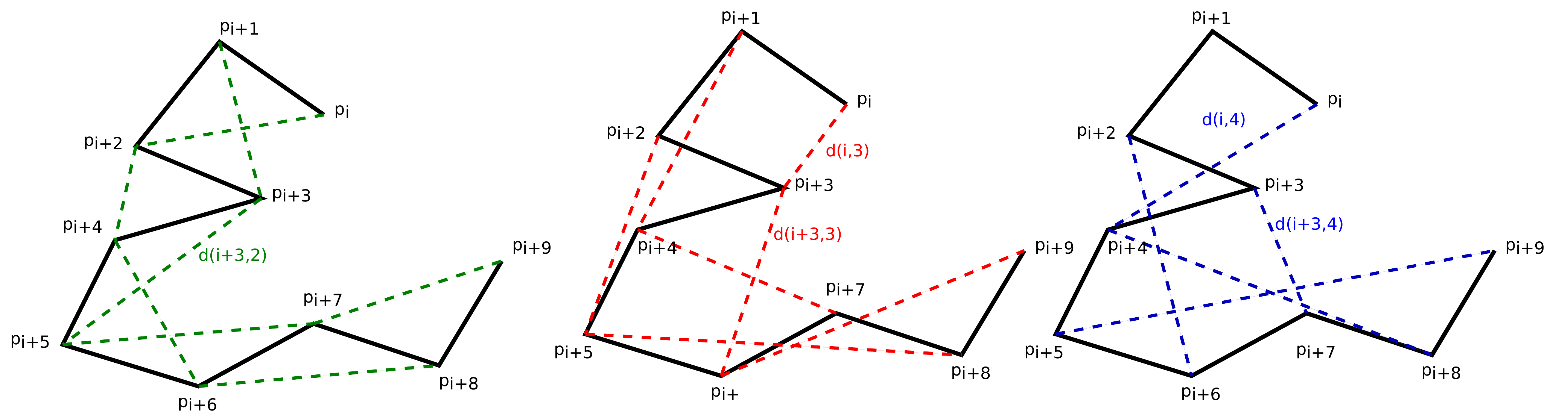}
  \end{center}
 \caption{In the above image, we have end to end distances marked for $k=2$, $k = 3$ and $k = 4$. We would take the average of their squares to find the average squared end to end distance for $k=2,3$ and $4$. }
\end{figure}

 \begin{theorem}
 \label{EtE}
 $\displaystyle A(k,n)=\frac{k(n-k)}{n-1}$
 \end{theorem} 

\begin{proof}
First we have that 

\begin{eqnarray*}
d_{P}(k,j) = \left|\left| \displaystyle \sum_{i=j}^{j+k} e_i \right|\right|^2 &=& (\displaystyle \sum_{i=j}^{j+k} e_i) \cdot (\displaystyle \sum_{i=j}^{j+k} e_i )\\
&=& \sum_{i=j}^{j+k} (e_i\cdot e_i)+2\sum_{j\leq i < m \leq j+k} (e_i \cdot e_m) \\
&=& k+2\sum_{j\leq i < m \leq j+k} (e_i \cdot e_m) 
\end{eqnarray*}  

We have shown that independent of $P$, $k$, $i$ and $j$, the average value of $\displaystyle (e_i \cdot e_m) = \frac{-1}{n-1}$. Therefore the average value of $d_{P}(k,j)$ is, independent of $P$, $i$ and $j$,  $\displaystyle  k+2\sum_{j\leq i < m \leq j+k}r_n$. This average is precisely $A(k,n)$. We conclude that 
\begin{eqnarray*}
A(k,n) & = &  k+2\sum_{j\leq i < m \leq j+k}r_n \\
& = & k + 2 \frac{(k-1)(k)}{2} \frac{-1}{n-1} \\
&=& \frac{k(n - k)}{n-1} \\
\end{eqnarray*} 

Therefore the average value of $\parallel v_k \parallel^2$ is $\displaystyle A(k,n) = \frac{k(n-k)}{n-1}$.\
\end{proof}

We can verify some key features of this identity immediately. At $k = 1$ and $k = n-1$, we have that the average squared end to end distance is $1$. Vertices one edge apart must be distance $1$ apart, as we would expect, so the theoretical average agrees with the physical reality. Also, the function is symmetric about $\frac{n}{2}$, as a segment of length $k$ and its complement, a segment of length $n-k$, should have the same end to end distance. Lastly, the function achieves its maximum at $\frac{n}{2}$, as after that point, on average, the end to end distance must decrease due to the closure condition. Witz showed through a different argument that squared end to end distance scales like  $\frac{k(n-k)}{n}$, which is close to our value \cite{Witz10}. This is an approximation though, as for $k = 1$, we have $\frac{n-1}{n} \neq 1$, though it is close for large values of $k$ and $n$.\\

As in \cite{DoiEdwards} and Lemma \ref{Cn}, for open chains, the average squared end to end distance for ideal chains is $n$ (or $nb^2$ where $b$ is the length of the segments.) The squared end to end distance of subsegments of length $k$ within an open chain of length $n$ will be identical to the squared end to end distance of an open chain of length $k$. Therefore, in the open case, the average squared end to end distance of a subsegment of length $k$ within a chain of length $n$ is $k$. As we can see in Figure \ref{etoecomp}, the average squared end to end distance of subsegments of an open chain and the average squared end to end distance subsegments of an ideal ring look very different for subsegments of very short lengths when $n =50$. We observe that for the average squared end to end distance of segments of length $k$ in open chains to be within $\frac{1}{100}^{\textrm{th}}$ of the average squared end to end distance of segments of length $k$ in ideal rings, we would need $ k  \approx \frac{n}{100}$ or smaller. For length $50$, as in the figure below, $k$ must be length $1$, which is a very strong restraint. 

\begin{figure}[!htp]

  \begin{center}
    \includegraphics[scale=.78]{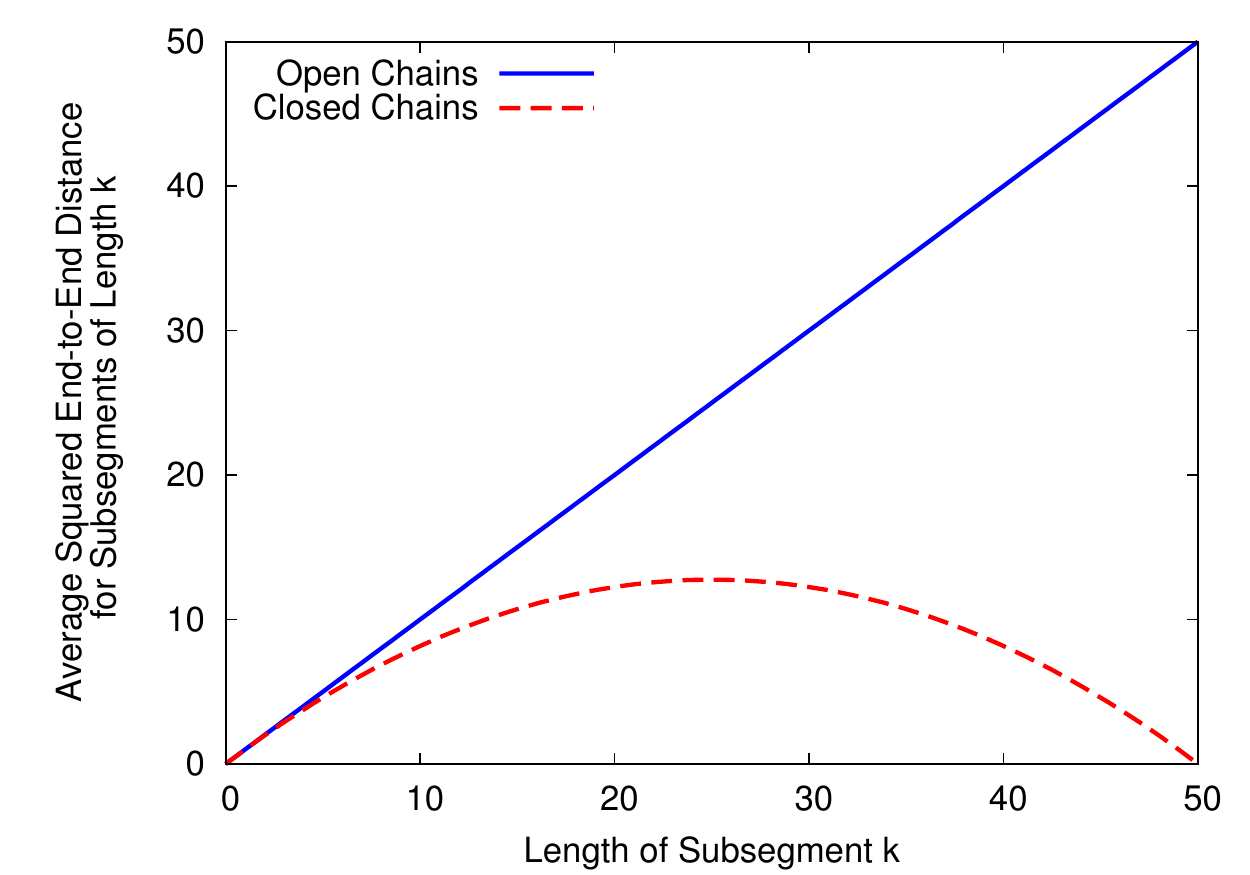}
  \end{center}
 \caption{Here we compare the theoretical average squared end to end distance of subsegments of length $k$ for open chains (blue) and ideal rings (red and dashed), both with total length $n = 50$.}
\label{etoecomp}   
\end{figure}

\subsection{Average Center of Mass} 

For the following definition and Lemma, it is helpful to recall that these ideal rings are based at the origin, $\displaystyle \sum_{i = 1}^n e_i = 0$, as we stipulated originally.\

\begin{definition}
Let $n \in \mathbb{N}$. Let $P\in \mathcal{P}_n$. Define $c_P$, the {\em center of mass of $P$}, to be $\displaystyle c_P = \frac{1}{n} \sum_{k=1}^n v_k $, the average of the vertices.  
\end{definition}

\begin{lemma}
\label{c_P}
For any $P \in \mathcal{P}_n$, $$\displaystyle \parallel c_P \parallel^2 =  \frac{1}{n^2} (\sum_{k=1}^n (n-k+1)^2 + 2 \sum_{1 \leq i < j \leq n} (n-j+1)(n-i+1)(e_i \cdot e_j)). $$
\end{lemma}

\begin{proof}
\begin{eqnarray*}
\displaystyle \parallel c_P \parallel^2 &=& (\frac{1}{n} \sum_{k=1}^n v_k \cdot \frac{1}{n} \sum_{k=1}^n v_k) \\
&=&  \frac{1}{n^2} ( \sum_{k=1}^n \sum_{i=1}^k e_i \cdot \sum_{k=1}^n \sum_{j=1}^k e_j)) \\
&=&  \frac{1}{n^2} (\sum_{k=1}^n (n-k+1)^2 + 2 \sum_{1 \leq i < j \leq n} (n-j+1)(n-i+1)(e_i \cdot e_j))
\end{eqnarray*}
\end{proof}

\begin{definition}
As above, let us denote the average of the set $\{ \parallel c_P \parallel^2 : P \in \mathcal{P}_n\}$ as $\parallel c_n \parallel^2$, the {\em average center of mass} of $\mathcal{P}_n$. 
\end{definition}

\begin{lemma}
\label{c_P4n}
$\displaystyle \parallel c_n \parallel^2 = \frac{n+1}{12} .$

\end{lemma}
 
\begin{proof}
From Lemma \ref{c_P}, we have that for any $P \in \mathcal{P}_n$, $$\displaystyle \parallel c_P \parallel^2 =  \frac{1}{n^2} (\sum_{k=1}^n (n-k+1)^2 + 2 \sum_{1 \leq i < j \leq n} (n-j+1)(n-i+1)(e_i \cdot e_j)).$$

We will replace $(e_i \cdot e_j)$ with $r_n$ to find $\parallel c_n \parallel^2$:
\begin{eqnarray*}
 \parallel c_n \parallel^2 &=&  \frac{1}{n^2} (\sum_{k=1}^n (n-k+1)^2 + 2 \sum_{1 \leq i < j \leq n} (n-j+1)(n-i+1)r_n).\\
 & = & \frac{1}{n^2} (\sum_{k=1}^n (n-k+1)^2) + \frac{2}{n^2} \sum_{j = 2}^n \sum_{i = 1}^{j-1}\left( (n-j+1)(n-i+1)\left( \frac{-1}{n-1}\right)\right).
\end{eqnarray*}

Simplifying the first term, we have $$\displaystyle \frac{1}{n^2} (\sum_{k=1}^n (n-k+1)^2) =  \frac{2n^2 + 3n +1}{6n}.$$ Likewise the second term, 
$$\displaystyle \frac{2}{n^2} \sum_{j = 2}^n \sum_{i = 1}^{j-1}\left( (n-j+1)(n-i+1)\left( \frac{-1}{n-1}\right)\right)$$
is simplified to 
$$\frac{-1}{n^2(n-1)} \sum_{j = 2}^n \left(j^3 - (3n+4)j^2 + (2n^2 + 7n +5)j - (2n^2+4n+2) \right).$$ 
The fact that we are summing from $j = 2$ rather than $j = 1$ complicates this calculation. To fix this, we change the sum so that it starts at $j = 1$, and at the end, we will subtract off the $j = 1$ term: 
$$1 - 3n-4 + 2n^2 + 7n +5 - 2n^2-4n-2 = 0.$$ 
Because the $j = 1$ term is $0$, we may change the summation to start at $j = 1$. This allows further simplification, and the second term is finally $$ \displaystyle -\left(\frac{3n^4 + 2n^3 -3n^2 - 2n}{12n^2(n-1)}\right).$$ 

Combining these two terms, we have 
\begin{eqnarray*}
\parallel c_P \parallel^2 & = & \frac{2n^2 + 3n +1}{6n} - \frac{3n^4 + 2n^3 -3n^2 - 2n}{12n^2(n-1)} \\
&=& \frac{(n+1)}{12}\\
\end{eqnarray*}

\end{proof} 

\subsection{Squared Radius of Gyration}

With the average center of mass defined, we may ask for the average squared difference between the center of mass and the vertices. This average difference is precisely the average squared radius of gyration. 

\begin{definition}
Let $P \in \mathcal{P}_n$ with vertices $v_1,...,v_n$. With $c_P$ as above, define the {\em squared radius of gyration of $P$} to be $$\displaystyle R_G^2(P) = \frac{1}{n} \sum_{k=1}^n  \parallel v_k - c_P \parallel^2.$$
\end{definition}

\begin{definition}
We will define the average of the set $\{ R_G^2(P) : P\in \mathcal{P}_n\}$ to be $R_{G,n}^2$, the {\em average squared radius of gyration} of $\mathcal{P}_n$.  
\end{definition}

\begin{theorem}
\label{ARG}
The average squared radius of gyration for all $P \in \mathcal{P}_n$ is $R_{G,n}^2 = \frac{n+1}{12}$.    
\end{theorem}

\begin{proof}
For any $P \in \mathcal{P}_n$, we have that 
\begin{eqnarray*}
\displaystyle R_G^2(P) &=&  \frac{1}{n} \sum_{k=1}^n \parallel v_k - c_P \parallel^2\\
&=& \frac{1}{n} \sum_{k=1}^n (\parallel v_k \parallel^2 - 2 (v_k \cdot c_P) + \parallel c_P \parallel^2) \\
&=& \frac{1}{n}\sum_{k=1}^n \parallel v_k \parallel^2 - 2 ( \sum_{k=1}^n\frac{v_k}{n} \cdot c_P) + \frac{1}{n}(n\parallel c_P \parallel^2) \\
&=&   \frac{1}{n}\sum_{k=1}^n \parallel v_k \parallel^2 - 2 ( c_P \cdot c_P)+ \parallel c_P \parallel^2\\
&=& \frac{1}{n}\sum_{k=1}^n \parallel v_k \parallel^2 -  \parallel c_P \parallel^2   
\end{eqnarray*}
From Theorem \ref{EtE} we have that  the average value of $ \parallel v_k \parallel^2 $ is $\frac{k(n-k)}{n-1} $. From Lemma \ref{c_P4n} we have that the average value of $\parallel c_P \parallel^2$ is $ \frac{n+1}{12}$. Replacing these, we can find the average squared radius of gyration over all $P \in \mathcal{P}_n$, that is, $R_{G,n}^2$ (rather than $R_G^2(P) $ as above.)
\begin{eqnarray*}
\displaystyle R_{G,n}^2&=& \left(\frac{1}{n}\sum_{k=1}^n \frac{k(n-k)}{n-1}\right) -  \frac{n+1}{12} \\
& = & \frac{n^2 - 1 }{12(n-1)} = \frac{n+1}{12} = \frac{n}{12} + \frac{1}{12}. 
\end{eqnarray*}
\end{proof}

Therefore the average squared radius of gyration scales like $\frac{n}{12}$, agreeing with Zimm and Stockmayer's estimate \cite{Zimm49}. 

\subsection{Average Squared Radius of Gyration of Subsegments of Length $k$}

The scaling of the average squared radius of gyration of ideal rings and chains is influenced by the presence of knotting \cite{Matsuda2003}.  As a consequence, one expects that the average squared radius of gyration of subsegments will also be affected by the presence or absence of knotting.

Determining a correlation between knotting and shape characteristics will be affected by the nature of knotting. If the knotted sections of the rings are small, involving relatively few edges, we may not see the influence of those edges in the squared radius of gyration, which can be obscured by the behavior of the unknotted complement. In order to detect local behavior, we study characteristics of shape that are local.

In open chains, the average squared radius of gyration of subsegments of length $k$ inside of larger chains of length $n$ is identical to the average squared radius of gyration of segments of length $k$. That is, the ambient chain has no effect of the shape of the subsegment. This is not true for ideal rings, where we will see that the closure constraint plays a pivotal role in local shape. However, these characteristics may play a pivotal role in identifying and characterizing the knotted portions of open chains. 

\begin{definition}
Let $P \in \mathcal{P}_n$. Define $P_{i,k}$ to be the translated subsegment of $P$ of length $k$ beginning with the $i^{\textrm{th}}$ edge vector, $e_i$. That is, $P_{i,k}$ is a segment starting at the origin, where the $j^{\textrm{th}}$ vertex is given by $\displaystyle \sum_{m = i}^{i+j} e_m$. For ease of notation, we will relabel the position and edge vectors so that the $j^{\textrm{th}}$ vertex is given by $v_j' = \displaystyle \sum_{m = 1}^{j} e_m'$.  
\end{definition}

So $P_{i,k}$ is isomorphic to a subsegment of $P$, though we've done some relabeling. Now we would like to find the center of mass and the squared radius of gyration of $P_{i,k}$. 

\begin{definition}
Let $\displaystyle c_{P_{i,k}} = \frac{1}{k} \sum_{j = 1}^{k} v_j' $, the {\em translated center of mass of $P_{i,k}$}. \
Then we define the {\em squared radius of gyration} of $P_{i,k}$ to be $$R_{G}^2(P_{i,k}) = \frac{1}{k} \displaystyle \sum_{j=1}^{k} \parallel v_j' - c_{P_{i,k}}\parallel^2.$$\
For $P \in \mathcal{P}$ and $k < n$, define the {\em average squared radius of gyration of subsegments of length $k$} as $$R_{G,k}^2(P) = \frac{1}{n} \displaystyle \sum_{j = 1}^{n}  R_{G}^2(P_{i,k}) .$$
\end{definition}

Now, we have the analogous definitions when we take the averages of $c_{P_{i,k}}$ and $R_{G,k}^2(P)$ over all $P \in \mathcal{P}_n$.

\begin{definition}
Let $c_{k,n}$ be {\em average center of mass of a translated subsegment of length $k$}, the average of $\{c_{P_{i,k}}:i \leq n\textrm{ and }P\in \mathcal{P}_n\}$. For some $n$, let $R_{G,n,k}^2$ be the {\em average squared radius of gyration of subsegment of length $k$} to be the average of $\{R_{G,k}^2(P): P \in \mathcal{P}_n\}$.  
\end{definition}
 
As in the last section, in order to compute $R_{G,n,k}^2$, we must first find $c_{k,n}$.  

\begin{lemma}
\label{later}
$\displaystyle \parallel c_{k,n} \parallel^2 = \frac{(2k^2 + 3k + 1)2n -3k(k+1)^2}{12k(n-1)}.$
\end{lemma}

\begin{proof}
Let us begin with some $\displaystyle \parallel c_{P_{m,k}} \parallel^2. $ 

\begin{eqnarray*}
\displaystyle \parallel c_{P_{m,k}} \parallel^2 &=& \frac{1}{k^2} (\sum_{j=1}^{k} v_j' \cdot \sum_{j=1}^{k} v_j')\\
&=& \frac{1}{k^2} (\sum_{j=1}^{k} \sum_{i = 1}^j e_i' \cdot \sum_{j=1}^{k} \sum_{i = 1}^j e_i')\\
&=& \frac{1}{k^2} \left(\sum_{i=1}^{k} (k - i + 1)^2 + 2 \sum_{j=2}^{k}\sum_{i = 1}^{j-1} (k - j + 1)(k - i + 1)(e_i' \cdot e_j')\right)\\
\end{eqnarray*}

Because $(e_i' \cdot e_j') = (e_{i+m}' \cdot e_{j+m}')$, the product is still independent of $i$ and $j$, and we can replace $(e_i' \cdot e_j')$ with $\displaystyle r_n  = \frac{-1}{n-1}$. This will let us take the average over all such $P_{m,k}$ to obtain $\parallel c_{k,n} \parallel^2$:
$$\displaystyle \parallel c_{k,n} \parallel^2 = \frac{1}{k^2} \left(\sum_{i=1}^{k} (k - i + 1)^2 + 2 \sum_{j=2}^{k}\sum_{i = 1}^{j-1} (k - j + 1)(k - i + 1)\left(\frac{-1}{n-1}\right)\right). $$

There are two sums to evaluate. The first is straightforward: $\displaystyle \sum_{i=1}^{k} (k - i + 1)^2= \frac{k(k+1)(2k+1)}{6}$
The second sum is $S = \displaystyle 2 \sum_{j=2}^{k}\sum_{i = 1}^{j-1} (k - j + 1)(k - i + 1)\left(\frac{-1}{n-1}\right).$ 

\begin{eqnarray*}
S & = & \frac{-2}{n-1} \sum_{j=2}^{k}\sum_{i = 1}^{j-1} (k - j + 1)(k - i + 1)\\
& = & \frac{-1}{n-1} \sum_{j=2}^{k}\left(j^3 - (3k +4)j^2 + (2k^2 +7k +5) j - 2(k+1)^2 \right)\\
\end{eqnarray*}

Evaluating $j^3 - (3k +4)j^2 + (2k^2 +7k +5) j - 2(k+1)^2$ at $j = 1$ we have $1 - 3k +4 + 2k^2 +7k +5 - 2k^2 - 4k -2 = 0$. Thus we can replace the lower bound of our sum, $j = 2$ with $j = 1$ with no penalty. 

\begin{eqnarray*}
S & = & \frac{-1}{n-1} \sum_{j=1}^{k}\left(j^3 - (3k +4)j^2 + (2k^2 +7k +5) j - 2(k+1)^2 \right)\\
& = & \frac{ 2k + 3k^2 - 2k^3 - 3k^4}{12(n-1)}
\end{eqnarray*}

Combining these terms, we have 
\begin{eqnarray*}
\parallel c_{k,n} \parallel^2 &=& \frac{1}{k^2}(\frac{k(k+1)(2k+1)}{6} +  \frac{ 2k + 3k^2 - 2k^3 - 3k^4}{12(n-1)})\\
& = & \frac{(2k^2 + 3k + 1)2n - 3k(k+1)^2}{12k(n-1)}
\end{eqnarray*}

\end{proof}

We note that when $k = n$, we have 
\begin{eqnarray*}
\displaystyle \parallel c_{n,n} \parallel^2 & = & \frac{(2n^2 + 3n + 1)2n - 3n(n+1)^2}{12n(n-1)}\\
& = & \frac{n^3  - n}{12n(n-1)} = \frac{n+1}{12}\\
\end{eqnarray*}
which agrees with Lemma \ref{c_P4n}.

\begin{theorem}
\label{ARG}
The average squared radius of gyration of a subsegment of length $k$, taken over all such subsegments in all $P \in \mathcal{P}_n$ is $R_{G,n,k}^2 = \frac{(k^2 - 1)(2n - k)}{12k(n-1)}$.    
\end{theorem}

\begin{proof}
For some $P_{i,k}$ we have 
\begin{eqnarray*}
R_{G}^2(P_{i,k}) &=& \frac{1}{k} \sum_{j=1}^{k} \parallel v_j' - c_{P_{i,k}}\parallel^2\\
&=& \frac{1}{k}\sum_{j=1}^{k}\parallel v_j'\parallel - 2 (\frac{1}{k}\sum_{j=1}^{k} v_j' \cdot c_{P_{i,k}}) + \frac{1}{k}\sum_{j=1}^{k}\parallel c_{P_{i,k}}\parallel^2\\
&=& \frac{1}{k}\sum_{j=1}^{k}\parallel v_j'\parallel - 2 ( c_{P_{i,k}}\cdot c_{P_{i,k}}) + \frac{1}{k}k\parallel c_{P_{i,k}}\parallel^2\\
&=& \frac{1}{k}\sum_{j=1}^{k}\parallel v_j'\parallel - \parallel c_{P_{i,k}}\parallel^2\\
\end{eqnarray*}

From Lemma \ref{later}, $\parallel c_{P_{i,k}}\parallel^2$ is, on average, $\displaystyle \parallel c_{k,n} \parallel^2 =\frac{(2k^2 + 3k + 1)2n - 3k(k+1)^2}{12k(n-1)}$. Likewise, average value of $\parallel v_j'\parallel^2$, the end to end distance of the segment $P_{i,k}$, is given in Theorem \ref{EtE}. By replacing $\parallel c_{P_{i,k}}\parallel^2$ and $\parallel v_j'\parallel^2$ with the averages for these values, we can find $R_{G,n,k}^2$:

\begin{eqnarray*}
R_{G,n,k}^2 &=& \frac{1}{k}\sum_{j=1}^{k} \frac{j(n-j)}{n-1} - \frac{(2k^2 + 3k + 1)2n - 3k(k+1)^2}{12k(n-1)}\\
&=&\frac{(k^2 - 1)(2n - k)}{12k(n-1)}\\
\end{eqnarray*}

\end{proof}

We note that when $k = n$, we have 
\begin{eqnarray*}
\displaystyle R_{G,n,n}^2 &=& \frac{(n^2 - 1)(n)}{12n(n-1)}\\
& = & \frac{n+1}{12}\\
\end{eqnarray*}
which agrees with Theorem \ref{ARG}.
 
Now for a fixed $n$, we have a function that returns the average squared radius gyration of a subsegment of length $k$. 

\subsection{Comparison Between Ideal Rings and Open Ideal Chains}

For comparison, consider ideal chains. From Lemma \ref{c_P}, which also holds for ideal chains, we know that for any $W \in \mathcal{W}_n$, the squared center of mass is given by $$\parallel c_W \parallel^2 = \frac{1}{n^2} \left(\sum_{k=1}^n (n-k+1)^2 + 2 \sum_{1 \leq i < j \leq n} (n-j+1)(n-i+1)(e_i \cdot e_j)\right).$$ When the average is taken over all ideal chains $W \in \mathcal{W}_n$, define the average ${c_{\mathcal{W}_n}}$. Likewise, let the squared radius of gyration of the ideal chain $W$ be $R_G^2(W)$, and let the average over all such ideal chains be ${R_{G,\mathcal{W}_n}^2}$. 

\begin{lemma}
 $R_{G,\mathcal{W}_n}^2 =  \frac{n^2 - 1}{6n}.$
 \end{lemma}
\begin{proof}
First, we find $\parallel c_{\mathcal{W}_n} \parallel^2$, which is simplified by that fact that in $\mathcal{W}_n$, the average value of $(e_i \cdot e_j)$ is  $0$ for $i \neq j$.
\begin{eqnarray*}
\parallel c_{_n} \parallel^2 & = & \frac{1}{n^2} \left(\sum_{k=1}^n (n-k+1)^2 + 2 \sum_{1 \leq i < j \leq n} (n-j+1)(n-i+1)(e_i \cdot e_j)\right)\\
& = & \frac{1}{n^2} \left(\sum_{k=1}^n (n-k+1)^2\right)\\
& =& \frac{2n^2 + 3n +1}{6n}. 
\end{eqnarray*}
As in the first part of Theorem \ref{ARG}, the squared radius of gyration of some open chain, $R_G^2(W)$, is 
$$R_G^2(W) = \frac{1}{n}\sum_{k=1}^n \parallel v_k \parallel^2 -  \parallel c_W \parallel^2.$$ Replacing $\parallel c_W \parallel^2$ with the average $\parallel c_{\mathcal{W}_n} \parallel^2$ and $\parallel v_k \parallel^2$ with its average value, $k$, from Lemma \ref{Cn}, we can find the average radius of gyration for open chains, $R_{G,{\mathcal{W}_n}}^2$:

\begin{eqnarray*}
R_{G,{\mathcal{W}_n}}^2 & = & \frac{1}{n}\sum_{k=1}^n k - \parallel c_{\mathcal{W}_n} \parallel^2 \\
& = &  \frac{1}{n} \left(\frac{n(n+1)}{2}\right) -  \frac{2n^2 + 3n +1}{6n} \\
& = & \frac{n^2 - 1}{6n}. 
\end{eqnarray*} 
\end{proof} 
 
For open chains, the squared radius of gyration of a subsegment of length $k$ is the same as for a chain of length $k$. So we may compare $R_{G,n,k}^2  =\displaystyle \frac{(k^2 - 1)(2n - k)}{12k(n-1)}$ and $\displaystyle R_{G,\mathcal{W}_n}^2  = \frac{k^2 - 1}{6k}$.

\begin{figure}[!htp]
  \begin{center}
    \includegraphics[scale=.78]{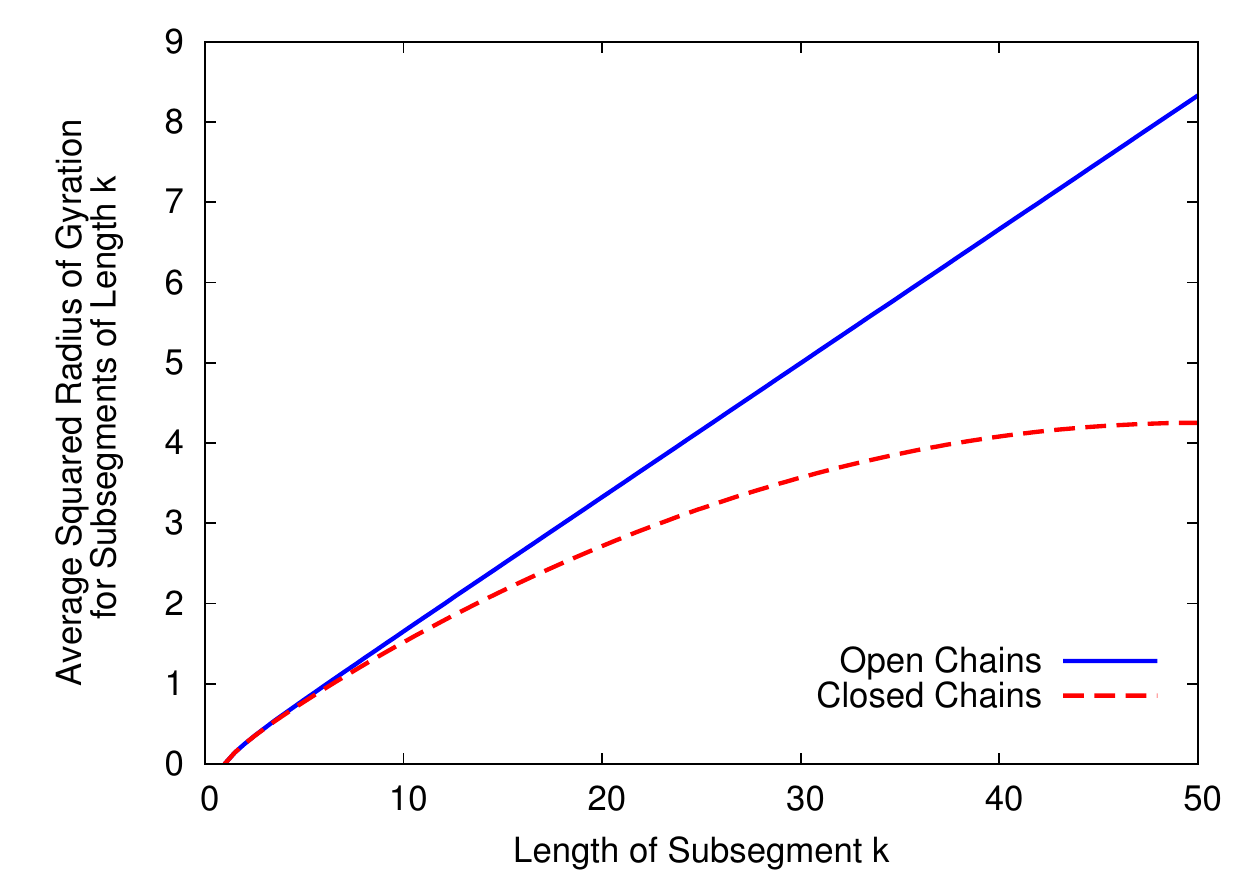}
  \end{center}
 \caption{Here we compare the average squared radius of gyration of subsegments of length $k$ for open chains (blue) and ideal rings (red and dashed).}
\end{figure}

As with squared end to end distance, we can see that the average squared radius of gyration of subsegments in an ideal ring is radically different from the average squared radius of gyration of subsegments in an open ideal chain. For ideal rings and chains of length $n$, in order to have the averages squared radius of gyration of a subsegment of length $k$ for ideal rings to be within $\frac{1}{100}$ of the squared radius of gyration of a subsegment of length $k$ in an ideal chain, we must have each other, we must have that the length of the subsegment considered, $k$, is less than $\frac{n}{50}$. So for $n=1000$, $k$ must be $20$ or less.

Recall, with squared end to end distance, for the averages to be within $1\%$ of each other, the segments had to be less than $\frac{1}{100}^{\textrm{th}}$ of the length, so $k$ had to be length $10$ or shorter if $n = 1000$. This suggests that the squared end to end distance characterization is more sensitive than squared radius of gyration of subsegments, as it is more affected by the closure condition. Once again, we observe that the local scale of an ideal ring differs significantly from that of a ideal chain at all but the smallest length scales.

\section{Experimental Methods}

Generating a random sample from $\mathcal{P}_n$ is much trickier than generating a random walk. With the latter, we can uniformly sample points edge vectors on the unit sphere. The complication in $\mathcal{P}_n$ is the closure constraint: we cannot just generate random walks and hope that they close. 

In order to randomly sample $\mathcal{P}_n$ we have two step process called the hedgehog method \cite{Ken11}. First, we find a starting point, then we use crankshaft rotations to sample the space. After some number of moves, our sample is independent of the starting point, and is a random element of $\mathcal{P}_n$.   

\subsection{Hedgehog Method}

The hedgehog method begins with selecting $n$ points uniformly on the unit sphere, and label them $e_1, ... , e_n$. Then, we add to that list each $e_i$'s negative, $-e_i$. Thus we have $2n$ edge vectors, $e_1,  -e_1, e_2, -e_2, ...,e_n,  -e_n$.

\begin{figure}[!htp]
  \begin{center}
    \includegraphics[scale=.60]{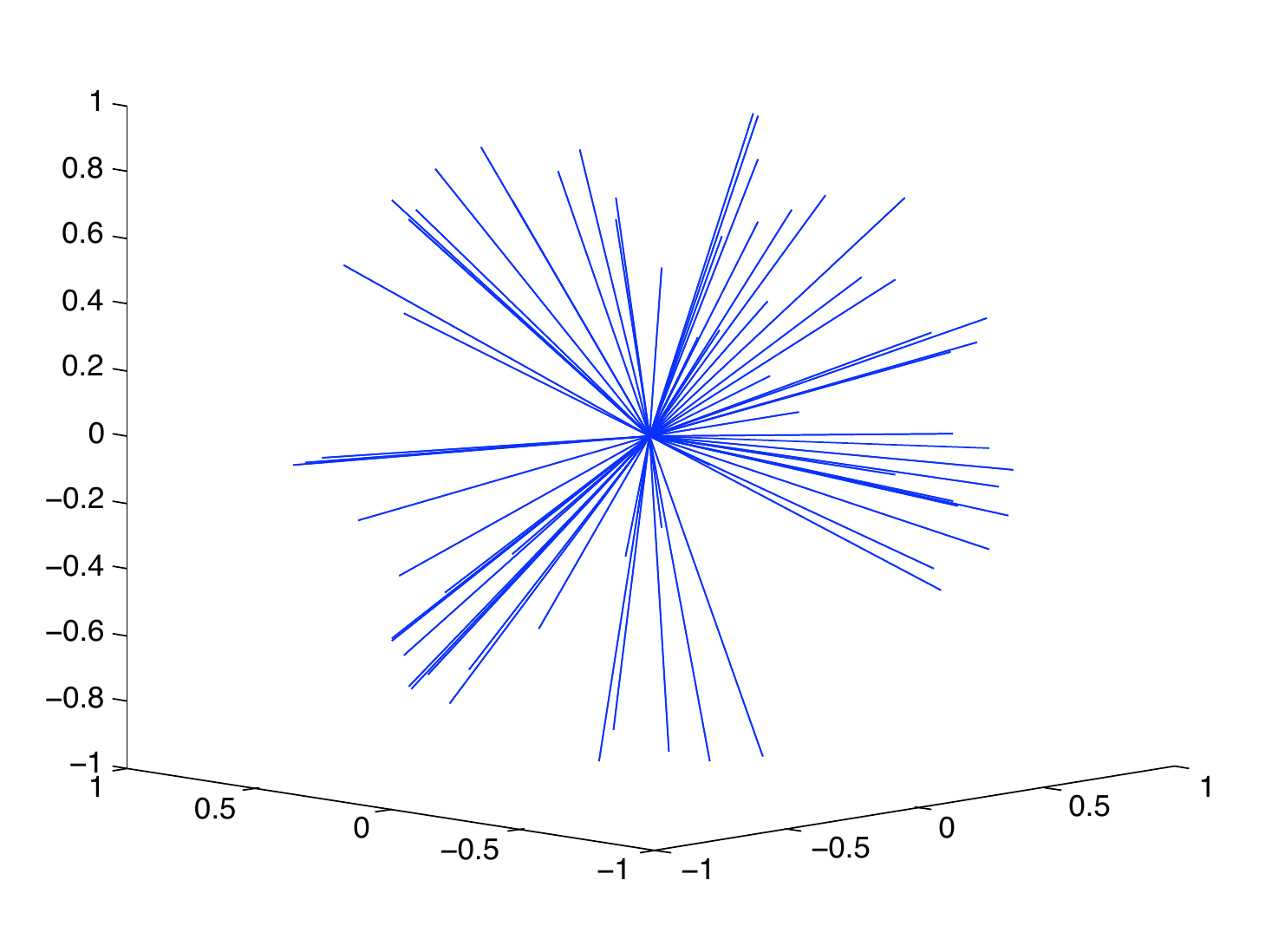}
  \end{center}
 \caption{The edge vectors $e_1,  -e_1, e_2, -e_2, ...,e_n,  -e_n$ are plotted above, hence this is called the hedgehog method.}
\end{figure}

We randomly permute these, getting a list of edge vectors $e_1', e_2', ... e_{2n}'$. Adding in each edge vector's negative ensures the all important closure condition is met.

\subsubsection{Crankshaft Rotations}

This is a good starting point, but for each edge vector, its exact opposite is also an edge vector, which is undesirable. 

We finish the hedgehog method by performing crankshaft rotations. We begin a crankshaft rotation by randomly selecting two non-parallel edge vectors $e_j$ and $e_k$. These are rotated about the axis determined by $e_j + e_k$, by a random angle $\theta$. The form of these rotations is given by the following:

\begin{eqnarray*}
e_j & \longmapsto & \frac{e_j + e_k}{2} + \frac{e_j - e_k}{2}\cos(\theta) + \frac{e_j \times e_k}{\parallel e_j + e_k \parallel}\sin(\theta)\\
e_j & \longmapsto & \frac{e_k + e_j}{2} + \frac{e_k - e_j}{2}\cos(\theta) + \frac{e_k \times e_j}{\parallel e_k + e_j \parallel}\sin(\theta)
\end{eqnarray*}

Because the sum $e_j+e_k$ is conserved, the modified sequence of edge vectors and vertices still satisfies the closure condition. 

Because any equilateral polygon can be deformed by a finite sequence of crankshaft rotations to the regular polygon, a finite series of crankshaft rotations will take us from one polygon to another random polygon \cite{Ken11}. We performed $6n$ crankshaft rotations on each sample polygon.  

\section{Experimental Results}

We use the above methods to randomly generate ideal rings. Numerically, we will have two primary foci: how quickly does the average for an ensemble converge to these theoretical values, and how will knotting affect these characteristics?  

\subsection{Convergence}

For various population sizes, $10^1$, $10^2$, $10^3$, $10^4$ and $10^5$, we generated that many $50$ edged polygons with $150$ crankshaft rotations. For each polygon, we found the squared radius of gyration, and took the average over the population. For each population size, we did this $10$ times. We then compared these averages to the theoretical average. We define $E(n)$ to be the difference between the experimental average and the theoretical average with $n$ samples. Figure \ref{conv} shows the convergence result. 

\begin{figure}[!htp]
  \begin{center}
    \includegraphics[scale=.45]{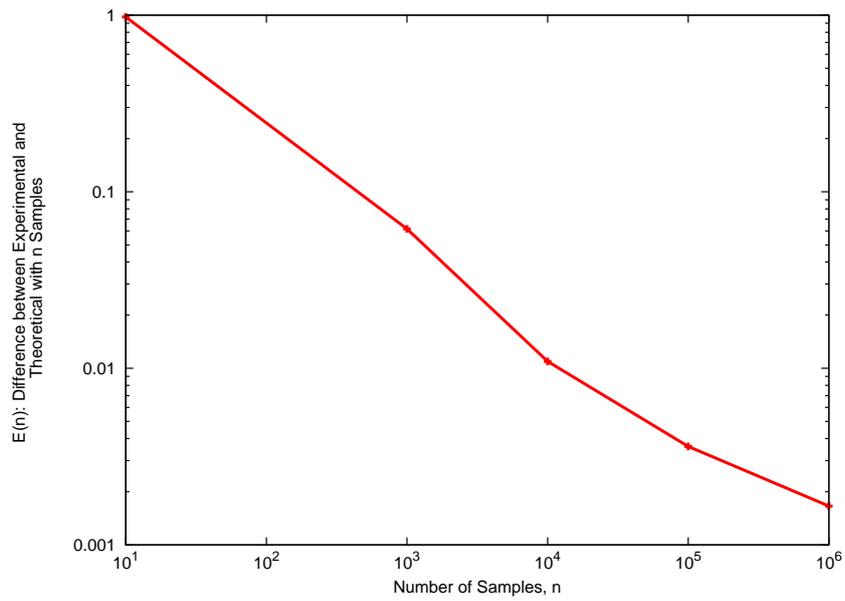}
  \end{center}
 \caption{For population sizes $10^1$, $10^2$, $10^3$, $10^4$ and $10^5$ we computed the average radius of gyration, and plotted it against the theoretical value for the same length. Note that this is a log scale plot of the data. }
   \label{conv}
\end{figure}

For population sizes of $10000$ and larger, we have excellent convergence to the theoretical average. Using linear regression, we can estimate that $E(n)\sim Cn^{-1.559}$, providing for excellent convergence for large $n$.

\subsection{Knotting}

Knotted ideal rings, specifically trefoils, were generated by first constructing ideal rings of length 50, then calculating their knot type, and saving those of the given knot type.

By the length of a knot in a ring we mean the length of the shortest subchain that contains the knot. The knot length is determined by examining subsegments of progressively longer length, starting at all possible locations. Each open segment is situated inside of a large sphere. A random collection of points on the surface of the sphere is selected. The ends of the segment are closed to each point, and the knot type of each of these closures is then calculated. This yields a spectrum of knot types, as in \cite{ marcone2005,  marcone2007, Ken05, millett2005, tubiana2011}. A segment is considered to be a trefoil if the closure is a trefoil with some tolerance (greater than $50$\%). The knotted portion is identified as the shortest segment which is a trefoil, for which the complement is unknotted.   

\begin{figure}[!htp]
  \begin{center}
    \includegraphics[scale=.78]{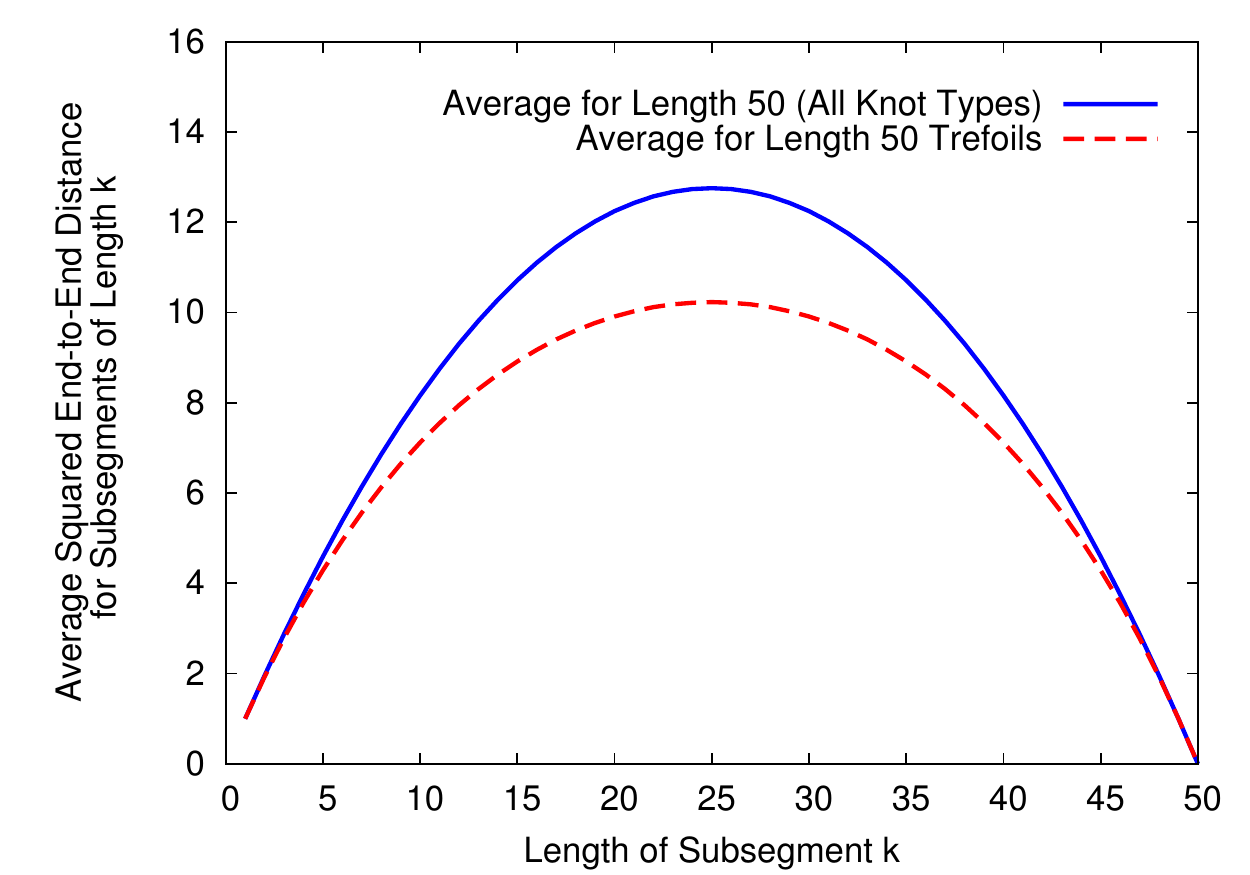}
  \end{center}
 \caption{Here we compare the average squared end to end distance for a phantom population of ideal rings (blue) and a population of randomly sampled trefoils (red and dashed).}
\end{figure}

Above we compare the average squared end to end distance for a phantom population of ideal rings and a population of randomly sampled trefoils. We can see that for length $50$, the average for the trefoils is smaller than the average for the whole space. This suggests that for length $50$, knotting compresses the polygon, making vertices closer together.

Looking at the maximum squared end to end distance, we may ask what length curve has the same maximum, $10.2291$. Solving, we have that a curve of $39.89$, approximately $40$, has the given maximum. That suggests that the average shortening caused by knotting is about $10$ edges of length, and that on average a trefoil of length $50$ has an effective length of $40$. 

The average length of the knotted portion of these trefoils is $16.4$. We predict that the difference between the length $10$ reduction we saw above and the average trefoil length, $16.4$, can be accounted for by examining the end to end distance of the knotted portion. As in Figure \ref{knot}, the difference could be explained by an average knotted portion of $16$ to $17$ edges with an average end to end distance around $6$.    

\begin{figure}[!htp]
  \begin{center}
    \includegraphics[scale=.5]{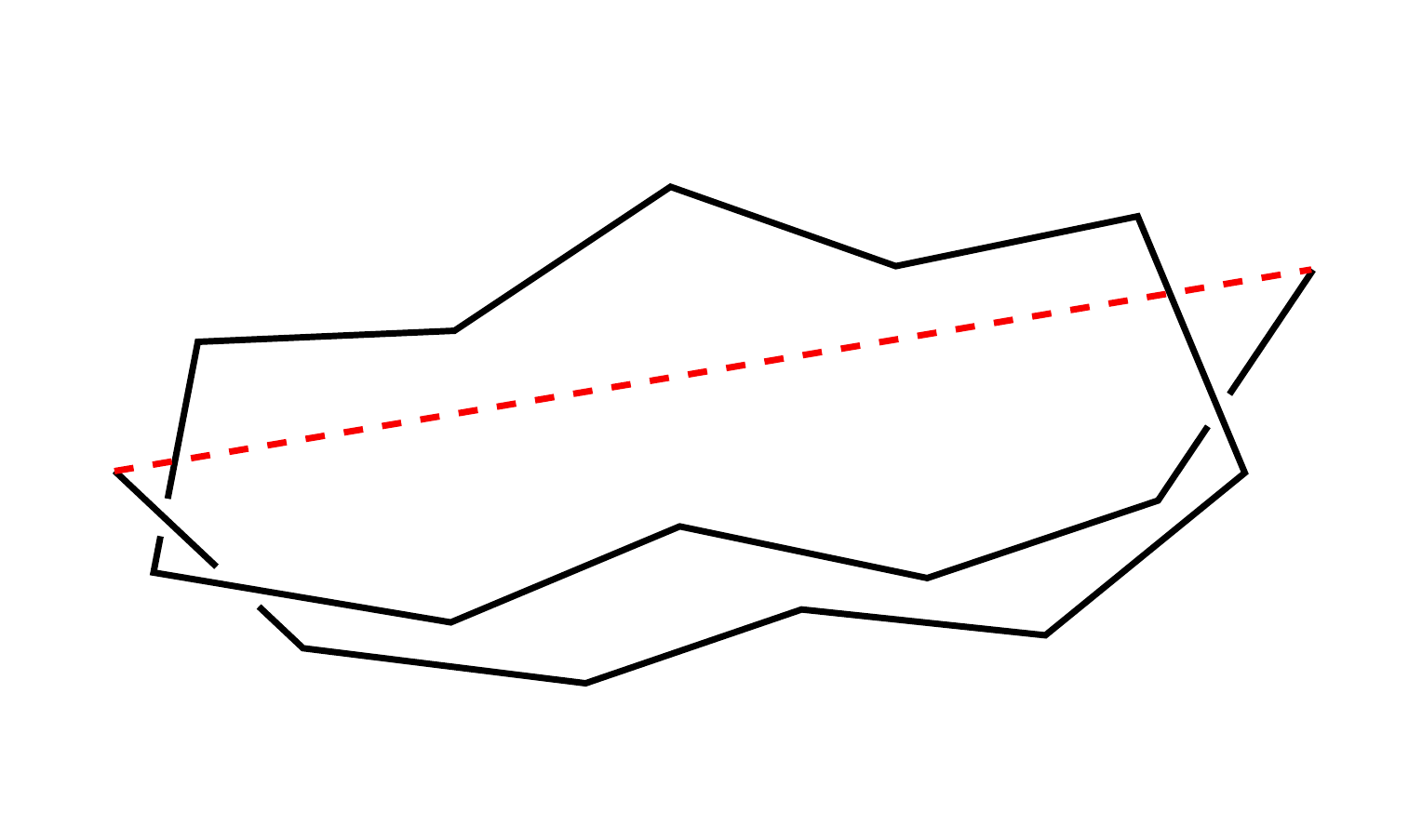}
  \end{center}
 \caption{A hypothetical average knotted section with knot length $16$ edges and end to end distance, (red and dashed), of $6$.}
\label{knot}
\end{figure}

\begin{figure}[!htp]
  \begin{center}
    \includegraphics[scale=.78]{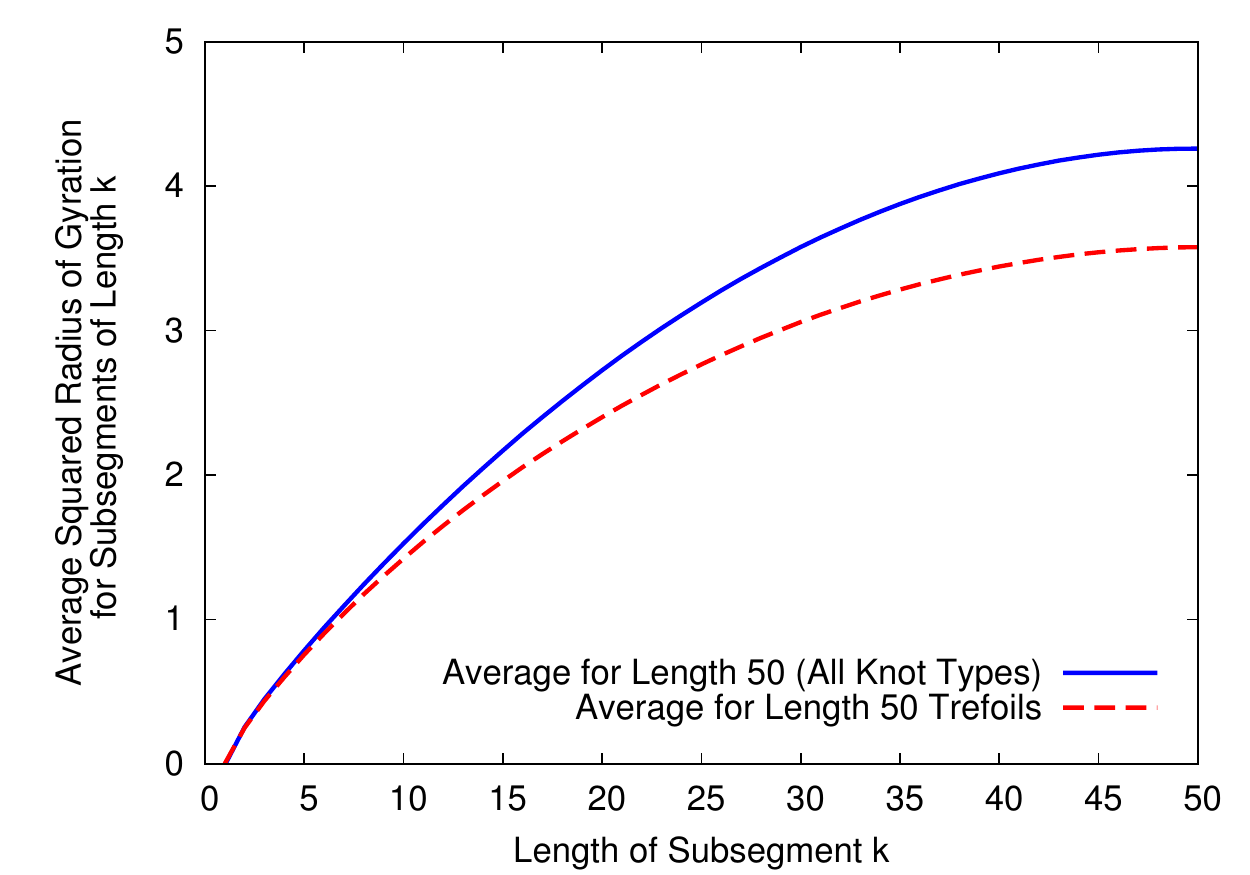}
  \end{center}
 \caption{Here we compare the average squared radius of gyration of subsegments of a phantom population of ideal rings (blue) and of a population of randomly sampled trefoils (red and dashed).  }
\end{figure}

As with squared end to end distance, the average for the trefoils is smaller than the average for the whole space. Again, this suggests that for length $50$, knotting compresses the polygon, making vertices closer together. We know that ideal trefoils, and indeed any collection of some fixed knot type, will for some small number of edges have mean squared radius of gyration less than the the mean for the total population, and for larger number of edges, they will have a greater average squared radius of gyration \cite{dobay03}. For example, the average squared radius of gyration of trefoils is smaller than the average squared radius of gyration of the whole population for lengths less than 175, and the opposite is true for lengths greater than 200 \cite{dobay03}. We expect similar behavior for squared end to end distance.

We can set the average squared radius of gyration of the population of trefoils, $3.5768$ equal to our function for squared radius of gyration, $\frac{n+1}{12}$. Solving for $n$ we have $n = 41.9216$, suggesting that the average shortening caused by knotting is about $8$ edges of length. We can compare this with the average shortening prediction from squared end to end distance, $10$. These differ by $4\%$ of the total length, $50$. 

\section{Conclusion}

These theoretical averages have many potential applications. Primarily, they can be used as a criterion to determine effectiveness of sampling methods. By comparing the the squared radius of gyration of subsegments or the squared internal end to end distance of polygons generated a given sampling method and the theoretical average, we can determine the effectiveness of the generation scheme and corroborate numerical simulations. Further, in the above generation method, we can use this convergence to determine how many crankshaft rotations are needed to sample the space of polygons uniformly, as in \cite{Ken11}. 

As the previous section highlights, squared end to end distance and squared radius of gyration of subsegments may be used to predict knot length, which is computationally expensive to calculate. This will allow us to examine the growth of the knot length as $n \rightarrow \infty$, to determine if average knot length is bounded, or grows proportionally with $n^{1/2}$ or $n$, allowing us to ascertain if knotting is strongly local, local or global on average 
\cite{diao2001, ercolini2007, katritch2000, marcone2005, marcone2007, micheletti2011, nakajima2011}.    

\section{Acknowledgments}

The authors would like to thank Eric Rawdon, both for his wonderful data and feedback on the paper. Eric was responsible for finding the trefoils and their knot lengths, a truly impressive computational task. 
 
\section{Bibliography}

\bibliographystyle{plain}
\bibliography{myrefs}

\end{document}